\ifx\documentclass\undefined
\documentstyle[12pt]{article}
\else
\documentclass[12pt]{article}
\fi

\makeatletter
\renewcommand{\theequation}{\thesection.\arabic{equation}}
\@addtoreset{equation}{section}
\def\eqnarray{%
\stepcounter{equation}%
\let\@currentlabel=\theequation
\global\@eqnswtrue
\global\@eqcnt\z@
\tabskip\@centering
\let\\=\@eqncr
$$\halign to \displaywidth\bgroup\@eqnsel\hskip\@centering
$\displaystyle\tabskip\z@{##}$&\global\@eqcnt\@ne
\hfil$\displaystyle{{}##{}}$\hfil
&\global\@eqcnt\tw@$\displaystyle\tabskip\z@{##}$\hfil
\tabskip\@centering&\llap{##}\tabskip\z@\cr}
\makeatother

\def\bbbz{{\mathchoice {\hbox{$\sf\textstyle Z\kern-0.4em Z$}}
{\hbox{$\sf\textstyle Z\kern-0.4em Z$}}
{\hbox{$\sf\scriptstyle Z\kern-0.3em Z$}}
{\hbox{$\sf\scriptscriptstyle Z\kern-0.2em Z$}}}}
\def\bbbq{{\mathchoice {\setbox0=\hbox{$\displaystyle\rm Q$}\hbox{\raise
0.15\ht0\hbox to0pt{\kern0.4\wd0\vrule height0.8\ht0\hss}\box0}}
{\setbox0=\hbox{$\textstyle\rm Q$}\hbox{\raise
0.15\ht0\hbox to0pt{\kern0.4\wd0\vrule height0.8\ht0\hss}\box0}}
{\setbox0=\hbox{$\scriptstyle\rm Q$}\hbox{\raise
0.15\ht0\hbox to0pt{\kern0.4\wd0\vrule height0.7\ht0\hss}\box0}}
{\setbox0=\hbox{$\scriptscriptstyle\rm Q$}\hbox{\raise
0.15\ht0\hbox to0pt{\kern0.4\wd0\vrule height0.7\ht0\hss}\box0}}}}
\def\bbbc{{\mathchoice {\setbox0=\hbox{$\displaystyle \rm C$}\hbox{\raise
0.06\ht0\hbox to0pt{\kern0.4\wd0\vrule height0.9\ht0\hss}\box0}}
{\setbox0=\hbox{$\textstyle\rm C$}\hbox{\raise
0.06\ht0\hbox to0pt{\kern0.4\wd0\vrule height0.9\ht0\hss}\box0}}
{\setbox0=\hbox{$\scriptstyle\rm C$}\hbox{\raise
0.06\ht0\hbox to0pt{\kern0.4\wd0\vrule height0.8\ht0\hss}\box0}}
{\setbox0=\hbox{$\scriptscriptstyle\rm C$}\hbox{\raise
0.06\ht0\hbox to0pt{\kern0.4\wd0\vrule height0.8\ht0\hss}\box0}}}}
\makeatletter
  \renewcommand{\theequation}{%
 \thesection.\arabic{equation}}
  \@addtoreset{equation}{section}
\makeatother
\newtheorem{theorem}{Theorem}[section]
\newtheorem{lemma}[theorem]{Lemma}
\newtheorem{corollary}[theorem]{Corollary}
\newtheorem{proposition}[theorem]{Proposition}
\newtheorem{remark}[theorem]{Remark}

\newsavebox{\toy}
\savebox{\toy}{\framebox[0.65em]{\rule{0cm}{1ex}}}
\newcommand{\QED}{\usebox{\toy}}
\def\nlni{\par\ifvmode\removelastskip\fi\vskip\baselineskip\noindent}
\newenvironment{proof}{\nlni\begingroup\it Proof.\rm}{
\endgroup\vskip\baselineskip}

\begin{document}
\setlength{\baselineskip}{15pt}
\title{
Elementary properties of circle map sequences
}
\author{Fumihiko Nakano
\thanks{Faculty of Science, 
Department of Mathematics and Information Science,
Kochi University,
2-5-1, Akebonomachi, Kochi, 780-8520, Japan.
e-mail : 
nakano@math.kochi-u.ac.jp}}
\date{}
\maketitle
\begin{abstract}
We study
the combinatorial and structural properties of the circle map sequences. 
We introduce 
an embedding procedure which gives a map 
$\Phi : \Omega \to W := \{ R, L \}^{\bf N}$
from the hull(closure of the set of translates) to the sequence of embedding operations through which we study the structure of 
$\Omega$. 
We also study 
the set of admissible words and classify them in terms of their appearance. 
\end{abstract}


\section{Introduction}
The circle map 
$v_0 \in \{ 0, 1 \}^{\bf Z}$
of rotation number 
$\alpha \in (0,1) \cap {\bf Q}^c$
is defined by
\[
v_0 (n) :=
1_{[1 - \alpha, 1)} (n \alpha \mbox{ mod }1), 
\quad
n \in {\bf Z}.
\]
We first recall its basic properties \cite{L}. 
Let
\begin{displaymath}
\alpha=
[a_1, a_2, \cdots]
:=
\displaystyle{\frac{1}
{a_{1}+\displaystyle{\frac{1}
{a_{2}+\displaystyle{{}_{\ddots}}}}}}, 
\quad
\alpha_n := [a_1, a_2, \cdots, a_n] = \frac {p_n}{q_n}
\end{displaymath}
be the continued fraction expansion of 
$\alpha$
and its rational approximation
($a_n \in {\bf N}$
and
$p_n$, $q_n$
are relatively prime). 
$p_n$
and 
$q_n$
satisfy
\begin{eqnarray}
\cases{
p_{n+1} = a_{n+1} p_n + p_{n-1} & \cr
q_{n+1} = a_{n+1} q_n + q_{n-1} & \cr
}
n \ge 0
\label{pq}
\end{eqnarray}
with
$(p_{-1}, q_{-1})=(1, 0)$, 
$(p_0, q_0) = (0, 1)$. 
Let 
$s_n \in {\cal A}^* := \bigcup_{n \ge 1} \{ 0, 1\}^n$
be the word given recursively by 
\begin{eqnarray*}
s_{-1} = 1, 
\;
s_0 = 0, 
\;
s_1 = s_0^{a_1-1} s_{-1}, 
\;
s_{n+1}= s_n^{a_{n+1}} s_{n-1}, 
\;
n \ge 1.
\end{eqnarray*}
Then 
$s_n$
has length 
$q_n$
and coincides with 
$( v_0 (1), v_0 (2), \cdots, v_0 (q_n) )$
and also coincides with 
$( v_0(-q_n+1), v_0(-q_n+2), \cdots, v_0(-1), v_0(0) )$
if 
$n$
is even ; in other words, 
$( v_0 (n) )_{n \ge 1}$
is the right limit of 
$s_n$
and 
$( v_0 (n) )_{ n \le 0}$
is the left limit of 
$s_{2n}$. 
$s_n$
($n \ge 1$)
can be written as
\[
s_n = \pi_n 
\cases{
(10) & ($n$: even) \cr
(01) & ($n$: odd) \cr
}
\]
where 
$\pi_n$
is a palindrome. 
If 
$\alpha$
is the reciprocal number of the golden number 
($\alpha = \frac {1}{\tau}
:=
\frac {\sqrt{5}-1}{2}
=
[1, 1, \cdots]$), 
then 
$s_1 = 1$, 
$s_2 = 10$, 
$s_3=101$, 
$s_4 = 10110, \cdots$
and 
$v_0$
is called the Fibonacci word which is thoroughly studied. 
We give 
the topology of pointwise convergence on 
$\{ 0, 1 \}^{{\bf Z}}$
(the product topology of the discrete topology on 
$\{ 0, 1\}$)
and let
\[
\Omega := 
\mbox{ closure of }
\{ v_0 (\cdot - m) \}_{m \in {\bf Z}}
\]
which is called the hull of 
$v_0$
and has the following representation. 
\begin{eqnarray}
\Omega &=& 
\{ v_{\theta} \}_{\theta \in {\bf T}}
\cup
\{ v'_0 (\cdot - m) \}_{m \in {\bf Z}}
\label{hull}
\\
v_{\theta} (n) &:=&
1_{[1 - \alpha, 1)} (n \alpha + \theta \mbox{ mod }1), 
\quad
\theta \in {\bf T}, 
\nonumber
\\
v'_0 (n) &:=&
1_{ (1 - \alpha, 1] } (n \alpha \mbox{ mod }1). 
\nonumber
\end{eqnarray}
Circle map sequences have the property that
(1) minimal complexity,
and
(2) aperiodic and balanced.
Actually, 
these three conditions are mutually equivalent \cite{MH}, and for that reason circle map sequences are sometimes called Sturmian sequences. 

The purpose of this paper 
is to study some elementary properties of 
$v_0$. 
In section 2, 
we consider Fibonacci word and introduce an ``embedding procedure" to construct elements of 
$\Omega$
to study the the combinatorial properties of 
$v_0$. 
This is
essentially a special case of the ``desubstitution" 
\cite{F, L}, 
which is studied well, 
though the formulation given here is slightly different.
We review
the relationship between this embedding and the two interval exchange dynamical system inheriting in the Fibonacci word, 
by which we study property of a measure on 
${\bf T}$
induced by a random embedding. 

In section 3, 
we consider the set of admissible words of 
$v_0$ 
and study how they distribute in 
$v_0$.
We classify 
them in terms of their occurrence in 
$v_0$
and compute their frequency. 
As is discussed (in more general context) in \cite{AB}, 
this classification gives us an alternative proof of the three-distance theorem\cite{Sos}.
In Appendix 1, 
we collect some basic properties of the embedding procedure. 
In Appendix 2, 
we discuss a combinatorial property of the circle map sequence which follows easily from the embedding procedure. 

In what follows, 
the definition of notation 
$|A|$
for a set 
$A$
should be clear from the context : 
it means the number of its elements if 
$A \subset {\bf Z}$, 
while it means the Lebesgue measure if 
$A \subset {\bf R}$. 
%

\section{An embedding procedure}
In this section, 
we consider the case of Fibonacci word : 
$a_n = 1$. 
We first 
define the ``embedding procedure". 
\subsection{Definition}
We first 
explain the motivation of considering this procedure.
Since we have 
$s_{n+1} = s_n s_{n-1}$
in Fibonacci word, 
it is possible to embed
$s_k$ 
to a larger 
$s_{k'}$
by either of the following two operations. 
\begin{eqnarray*}
(1) \quad R:\qquad
&s_n&
\\
\mapsto s_{n+1} := &s_n& s_{n-1},
\\
(2) \quad L :\qquad
&s_{n}&
\\
\mapsto s_{n+1} &s_{n}& =: s_{n+2}, 
\end{eqnarray*}
After infinitely many operations, we will have an element of 
$\Omega$. 
The converse 
will turn out to be true : every 
$v \in \Omega$
is obtained by this procedure.
Utilizing this fact, 
we would like to consider an analogue of the ``up-down generation" in the construction of the Penrose tiling. 
To define it properly, 
we first recall the results in 
\cite{Damanik-Lenz}
which applies to any circle map sequences. 
The 
$(n-1, n)$-partition
is the non-overlapping covering of a sequence 
$\{ v(n) \}_{n \in {\bf Z}}$
by two words 
$s_{n-1}$, $s_n$. 
\begin{lemma}\cite{Damanik-Lenz}
\label{partition}
For any 
$n \ge 0$, 
$v \in \Omega$
has unique 
$(n-1, n)$-partition.
\end{lemma}
\begin{corollary}\cite{Damanik-Lenz}
\label{corollary}
In the 
$(n-1, n)$-partition of 
$v \in \Omega$, \\
(1)
$s_{n-1}$
does not appear consecutively
($s_{n-1}$
is always isolated)\\
(2)
$s_n$
always appears 
$a_{n+1}$
or 
$(a_{n+1} + 1)$
times successively. 
\end{corollary}
Let 
\[
W := 
\{ (O_1, O_2, \cdots ) 
\; | \; 
O_j = R \mbox{ or } L 
\}
=
\{ R, L \}^{{\bf N}}.
\]
For given 
$v \in \Omega$, 
we construct the sequence 
$(O_1, O_2, \cdots) \in W$
of operations by the following procedure. 

(i)
When 
$v(0) = 1$, 
$v(0)$
is covered by 
$s_1$
in the 
$(0,1)$-partition. 
Set 
$O_1 = R$.
When 
$v(0) = 0$, 
$v(0)$
is covered by 
$s_2$
in the $(1,2)$-partition, for we have 
$(v_0(-1), v_0(0), v_0(1)) = (1, 0, 1)$.
Set 
$O_1 = L$.
\\
(ii)
Suppose
$v(0)$
is covered by a block 
$\underline{s_n}$
in the $(n-1, n)$-partition after the $k$-th step. 
If we find 
$s_{n-1}$
in the right to 
$\underline{s_n}$
in the $(n-1, n)$-partition, then 
$v(0)$
is covered by 
$s_{n+1}$
in the $(n, n+1)$-partition. 
In this case 
we regard that 
the block 
$\underline{s_n}$
containing 
$v(0)$
grow up to 
$s_{n+1}$
by putting
$s_{n-1}$
to its right end, so that we set 
$O_{k+1} = R$. 
\begin{eqnarray*}
&&\mbox{\fbox{　 $\underline{s_n}$　 }}
\\
&&\quad\downarrow R
\\
&&\mbox{\fbox{　 $s_n$　 }}\mbox{\fbox{$s_{n-1}$}}
\\
&&\qquad ||
\\
&&\mbox{\fbox{　　 $s_{n+1}$　　}}
\end{eqnarray*}
If we find 
$s_n$
in the right to 
$\underline{s_n}$, 
then 
$v(0)$
is still covered by 
$s_n$
in the 
$(n, n+1)$-partition, and 
is then covered by 
$s_{n+2}$
in the 
$(n+1, n+2)$-partition. 
In this case, 
we regard that the block 
$\underline{s_n}$
containing 
$v(0)$
grow up to 
$s_{n+2}$
by putting 
$s_{n+1}$
to its left end, so that we set 
$O_{k+1} = L$.
\begin{eqnarray*}
\mbox{\fbox{　 $s_n$　 }}
\mbox{\fbox{$s_{n-1}$}}
\mbox{\fbox{　 $\underline{s_n}$　 }}&&\mbox{\fbox{　 $s_n$　 }}
\\
L\downarrow\qquad&&
\\
\mbox{\fbox{　　 $s_{n+1}$　　 }}\mbox{\fbox{　 $s_n$　 }}&&
\\
||\qquad\qquad &&
\\
\mbox{\fbox{　　　　 $s_{n+2}$　　　　 }}&&
\end{eqnarray*}
In other words, 
if we find 
$s_{n-1}$
in the right to 
$\underline{s_n}$
in the $(n-1, n)$-partition, then we set 
$O_{k+1} = R$; 
otherwise we find 
$s_{n+1}$
in the left to 
$\underline{s_n}$
in the 
$(n+1, n+2)$-partition, and we set 
$O_{k+1} = L$. 
Hence we have defined a map 
\[
\Phi : \Omega \to W
\]
\begin{remark}
It is possible 
to define this map for any circle map sequences. 
In the n-th level, 
the embedding procedure is given by 
\begin{eqnarray*}
&&R_{(n,k)} : 
s_n
\mapsto s_n^{a_{n+1}} s_{n-1}, 
\quad
k= 1, 2, \cdots, a_{n+1}
\\
&&L_{n} : 
s_n
\mapsto s_{n+1}^{a_{n+2}} s_{n}
\end{eqnarray*}
$R_{(n,k)}$
means to embed 
$s_n$
to the 
$k$-th
$s_n$
in 
$s_{n+1}=s_n^{a_{n+1}} s_{n-1}$. 
This method also applies to the 
period-doubling sequence which is the fixed point of the substitution : 
$1 \mapsto 10$, $0 \mapsto 11$.
\end{remark}
\subsection{the inverse map}
To see
$\Phi$
is surjective and to find the subset of 
$\Omega$
on which
$\Phi$
is one to one, we study how to reconstruct 
$v \in \Omega$
for given 
$(O_1, O_2, \cdots ) \in W$
($O_j = R \mbox{ or } L$). 

$O_1 = R$ : 
Set 
$v(0) = 1$.
Then
$v(0)$
is covered by 
$s_1$
in the 
$(0,1)$-partition.

$O_1 = L$ : 
Set 
$v(0) = 0$.
Then we have 
$(v(-1), v(0), v(1)) = (1, 0, 1)$
so that 
$v(0)$
is covered by 
$s_2$
in the 
$(1, 2)$-partition.

After the 
$k$-th step, suppose that 
$v(0)$
is covered by 
$\underline{s_n}$
in the 
$(n-1,n)$-partition. 

$O_{k+1} = R$ : 
we put 
$s_{n-1}$
to the right end of  
$\underline{s_n}$
in the 
$(n-1, n)$-partition.
\begin{eqnarray*}
&&
\mbox{\fbox{　 n　 }}
\\
&&\quad
\downarrow R
\\
&&
\mbox{\fbox{　 n 　}}\mbox{\fbox{ n-1 }}
\\
&&
\qquad\quad
||
\\
&&
\mbox{\fbox{　　n+1　 　}}
\end{eqnarray*}
Then
$v(0)$
is covered by 
$s_{n+1}$
in the 
$(n, n+1)$-partition.

$O_{k+1} = L$ : 
we put 
$s_{n+1}$
to the left end of  
$\underline{s_n}$
in the 
$(n, n+1)$-partition.
\begin{eqnarray*}
\mbox{\fbox{ 　n　 }}&&
\\
L\downarrow\quad&&
\\
\mbox{\fbox{ 　　n+1 　　}}\mbox{\fbox{　 n 　}}&&
\\
||\qquad\qquad
&&
\\
\mbox{\fbox{　　　　n+2　　　　　}}&&
\end{eqnarray*}
Then
$v(0)$
is covered by 
$s_{n+2}$
in the 
$(n+1, n+2)$-partition.
We remark that, when 
$v(0)$
is covered by 
$\underline{s_n}$
in the 
$(n-1, n)$-partition, 
a number of letters has been further determined 
to the right of that and thus, in most cases, repeating this procedure determines a bi-infinite sequence
$( v(n) )_{n \in {\bf Z}}$. 
In fact, 
we always find
$\pi_{n+1}$
to the next to 
$\underline{s_n}$, 
since we have either 
$\underline{s_n} s_{n-1}s_n$
($O_{k+1}=R$)
or
$\underline{s_n} s_n s_{n-1}$
($O_{k+1}=L$)
in the 
$(n-1, n)$-partition. 
Because
$s_{n-1} \pi_n = \pi_{n+1}$, 
they are equal to either
$\underline{s_n} \pi_{n+1} (10)$
or 
$\underline{s_n} \pi_{n+1} (01)$.
However if
$O_j = R$
for large 
$j$, 
we have a semi-infinite sequence : 
$(v(n))_{n \ge -N}$ 
for some 
$N$, and 
$(v(n))_{n \le -N-1}$
is not determined. 
In this case 
$(v(n))_{n \ge -N}$
is equal to a translation of 
$(v_0(n))_{n \ge 1}$ : 
$v(-N + n-1) = v_0(n)$, $n \ge 1$. 
So by (\ref{hull}) we set either 
$( v(-N-2), v(-N-1) ) = (1, 0)$
or 
$(0,1)$
and further set
$v(-N-n) = v_0(n-2)$
for 
$n \ge 3$
so that we obtain an element of 
\[
\Omega_R := \{
v_0(\cdot + m), \; v'_0 (\cdot + m)
\; | \; 
m \ge 1 
\}.
\]
Hence, 
$\Phi$
is two to one on 
$\Omega_R$
and one to one elsewhere. 
Under the topology 
of the pointwise convergence on 
$\Omega$
and 
$W$, 
$\Phi$
and
$(\Phi : \Omega_R^c \to \Phi(\Omega_R^c) )^{-1}$
are continuous. 
$\Phi$
has an unique fixed point 
$f :=(L,R,R,L,R,L,R,R,L, \cdots)$
if we identify 
$R, L$
with 
$1, 0$
respectively and 
$v(n)$
with 
$O_{n+1}$. 
\begin{remark}
By this method
we see the correlation (constraint condition) of letters between different sites. 
In fact, if 
$n$
is even, both
$(10) s_n \pi_{n+1} (10)$
and 
$(01) s_n \pi_{n+1} (10)$
are allowed while only 
$(01) s_n \pi_{n+1} (01)$
is possible
(for odd
$n$, 
exchange
$(10)$
with
$(01)$). 
\end{remark}
%
\subsection{Relation to the division of intervals in ${\bf T}$}
Let 
$\Psi : {\bf T} \to \Omega$
be the map 
$\theta \in {\bf T} \mapsto v_{\theta} \in \Omega$.
We consider 
the inverse image of the cylinder set of 
$\Omega$ : 
e.g., 
\begin{equation}
\Psi^{-1}
(\{ v(0) = 1\})
=\left[ \frac {1}{\tau^2}, 1 \right),
\quad
\Psi^{-1}
(\{ v(0) = 0\})
=\left[ 0, \frac {1}{\tau^2} \right).
\label{sharp}
\end{equation}
If we go further, 
each interval is divided into two intervals with ratio 
$\tau : 1$. 
\begin{eqnarray*}
&&\Psi^{-1}(
\{
(v(0), v(1), v(2)) = (1,1,0) 
\})
=
\left[ 1-\frac {1}{\tau^3}, 1 \right),
\\
&&
\Psi^{-1}(
\{
(v(0), v(1), v(2)) = (1,0,1) 
\})
=
\left[ \frac {1}{\tau^2}, 1-\frac {1}{\tau^3} \right),
\\
&&\Psi^{-1}(
\{
(v(0), v(1), v(2), v(3)) = (0, 1,1,0) 
\})
=
\left[ \frac {1}{\tau^4}, \frac {1}{\tau^2} \right),
\\
&&
\Psi^{-1}(
\{
(v(0), v(1), v(2), v(3)) = (0, 1,0,1) 
\})
=
\left[ 0, \frac {1}{\tau^4} \right).
\end{eqnarray*}
Similarly, 
we consider 
$\Psi^{-1}(A_n)$
for 
$A_n = \{ v \in \Omega \; | \; 
v(0) = a_0, 
v(1) = a_1, \cdots, v(n) = a_n\}$
which corresponds to the two interval exchange dynamical system given by (\ref{sharp}).
As
$n$
becomes large, we have many intervals 
whose endpoints belong to
\[
D_- = 
\{x \; | \;  x \equiv n \alpha \pmod 1, 
\;
n = 0, -1, -2, \cdots  \}
\]
Since 
the induced system given by the first return map 
to each small interval is again 
the two interval exchange, 
each new interval is given by dividing 
each intervals into two ones with ratio 
$\tau : 1$, 
with the longer one has the previous dividing point as one of its endpoints. 

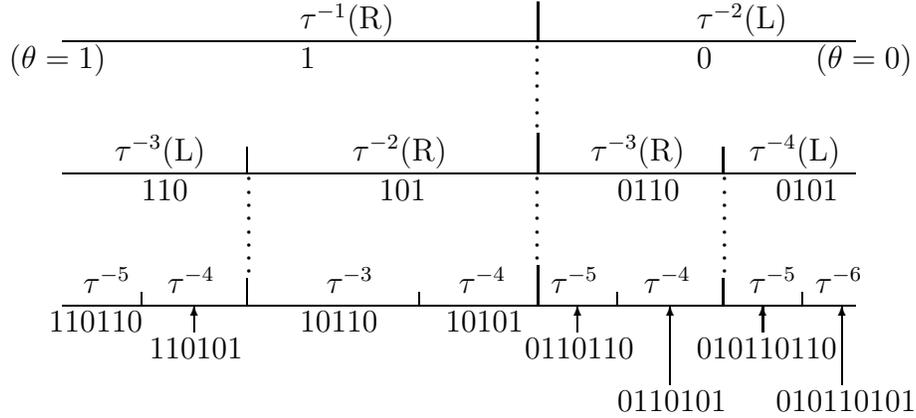
\vspace*{3em}
\begin{picture}(400, 100)
\put(30,0){
\put(0,100){\line(1,0){300}}
\put(0,50){\line(1,0){300}}
\put(0,0){\line(1,0){300}}
\put(180, 100){\line(0,1){15}}
\multiput(178, 95)(0, -5){7}{$\cdot$}
\multiput(178, 45)(0, -5){7}{$\cdot$}
\put(180, 50){\line(0,1){15}}
\put(180, 0){\line(0,1){15}}
\put(70, 50){\line(0,1){10}}
\multiput(69, 45)(0, -5){8}{$\cdot$}
\put(250, 50){\line(0,1){10}}
\multiput(249, 45)(0, -5){8}{$\cdot$}
\put(70, 0){\line(0,1){10}}
\put(250, 0){\line(0,1){10}}
\put(30, 0){\line(0,1){5}}
\put(135, 0){\line(0,1){5}}
\put(210, 0){\line(0,1){5}}
\put(280, 0){\line(0,1){5}}
\put(90, 105){$\tau^{-1}$(R)}
\put(240, 105){$\tau^{-2}$(L)}
\put(90, 90){$1$}
\put(240, 90){$0$}
\put(-20,90){($\theta=1$)}
\put(285,90){($\theta=0$)}
\put(20, 55){$\tau^{-3}$(L)}
\put(30, 40){$110$}
\put(110, 55){$\tau^{-2}$(R)}
\put(120, 40){$101$}
\put(200, 55){$\tau^{-3}$(R)}
\put(210, 40){0110}
\put(260, 55){$\tau^{-4}$(L)}
\put(270, 40){0101}
\put(8,5){$\tau^{-5}$}
\put(-5,-10){110110}
\put(40,5){$\tau^{-4}$}
\put(33,-20){110101}
\put(50, -10){\vector(0,1){10}}
\put(100,5){$\tau^{-3}$}
\put(90,-10){10110}
\put(150,5){$\tau^{-4}$}
\put(145,-10){10101}
\put(185,5){$\tau^{-5}$}
\put(175,-20){0110110}
\put(195, -10){\vector(0,1){10}}
\put(220,5){$\tau^{-4}$}
\put(210,-40){0110101}
\put(230, -30){\vector(0,1){30}}
\put(260,5){$\tau^{-5}$}
\put(240,-20){010110110}
\put(265, -10){\vector(0,1){10}}
\put(285,5){$\tau^{-6}$}
\put(270,-40){010110101}
\put(295, -30){\vector(0,1){30}}
}
\end{picture}

\vspace*{5em}

The operations 
$R$, $L$
correspond to those division of intervals in the following way \cite{F}. 
\begin{theorem}
\label{taiou}
The operation
$R$
(resp. $L$)
corresponds to creating the longer (resp. smaller)  interval.
\end{theorem}
\begin{proof}
Since 
the division of intervals corresponds to the words
$\underline{s_n}\pi_{n+1} (10)$
or 
$\underline{s_n}\pi_{n+1} (01)$, 
under the mapping 
$\Psi : {\bf T} \to \Omega$, 
it corresponds either to 
$R$
or
$L$. 
It then suffices to note that 
$L$
creates the word with the same ending of the original one, while 
$R$
creates the word with the opposite ending : 
$\cdots (01) \stackrel{L}{\to} \cdots (01)$, 
$\cdots (01) \stackrel{R}{\to} \cdots (10)$.
\QED
\end{proof}
\begin{remark}
If 
$\alpha \ne \frac {1}{\tau} (= \frac {\sqrt{5}-1}{2})$, 
we do not have such a simple relation except for quadratic numbers. 
In fact, 
we have many types 
$R_{(n, k)}$'s 
of embedding operations for general 
$\alpha$
and the induced system given by the first return map is not the two interval exchange in general. 
\end{remark}
%
%
\begin{remark}
\label{theta}
For given 
$w = (O_1, O_2, \cdots) \in W$, 
we can compute the corresponding 
$\theta = (\Phi \circ \Psi)^{-1}(w)$
as follows. 
\begin{eqnarray*}
\theta &=& \sum_{n=0}^{\infty} d_n, 
\\
d_0 &=& 1, \; d_1 = - \frac {1}{\tau}, 
\;
d_{n+1}
=
(-1)^{a_n + 1}
\left( \frac {1}{\tau} \right)^{a_n + 1} 
\left( \frac {1}{\tau^2} \right)^{b_n}, 
\quad
n \ge 1
\end{eqnarray*}
where
$a_n := \sharp \{ 1 \le k \le n \; : \; O_k = R \}$, 
$b_n := \sharp \{ 1 \le k \le n \; : \; O_k = L \}$.
This is equivalent to represent 
$\theta \in {\bf T}$
in terms of the sum of 
$\{ \frac {1}{\tau^k} \}_{k \ge 1}$.
\end{remark}
%
\begin{remark}
For 
$w = 
(w_1, w_2, \cdots, w_{n-1}, w_n) \in {\cal A}^*$, 
let 
$w^{-1} :=
(w_n, w_{n-1}, \cdots, w_2, w_1)$
be its mirror image.
Then 
$t_n := s_n^{-1}$ 
satisfies 
$
t_n = \cases{
(01) \pi_n & $(n : even)$ \cr
(10) \pi_n & $(n : odd)$ \cr
}
$
and 
$t_{n+1} = t_{n-1} t_{n}$.
Since 
$v \in \Omega \Longleftrightarrow v^{-1} \in \Omega$, 
$v \in \Omega$
always has 
$(n-1, n)$-partition by 
$t_n$ 
so that we can define embedding
$R'$, $L'$
by using 
$t_n$'s 
in the same way as $R$, $L$
(we set 
$v(-1)$
as the starting point). 
We have analogue of Theorem \ref{taiou}, and 
for 
$n$ 
even both 
$(01) \pi_{n+1} t_n (10)$
and 
$(01) \pi_{n+1} t_n (01)$
are possible while only 
$(10) \pi_{n+1} t_n (10)$
is allowed.
\end{remark}
%

\subsection{a measure induced by the random embedding}
Let 
$m$
be a measure on
$\{ R, L \}$
with 
$m(\{ R \}) = p \in (0,1)$, 
$m(\{L \}) = q :=1 - p$
and let 
${\bf P} := \otimes_{{\bf N}} m$.
In this section we study the measure 
$\mu$
on 
${\bf T}$
induced by the mapping 
$\Phi \circ \Psi : {\bf T} \to W$.
This may be 
regarded as an analogue of the Bernoulli convolution problem 
\cite{PSS}.  
Since
${\bf P}(\{ O_j = R \mbox{ for large  }j \}) =
{\bf P}(\{ O_j = L \mbox{ for large  }j \})= 0$, 
$\mu$
is a probability measure. 
It is easy to see that 
$\mu$
does not have atoms.
\begin{theorem}
(1)
If 
$p = \frac {1}{\tau}$, $q = \frac {1}{\tau^2}$, 
$\mu$
is equal to the Lebesgue measure.\\
(2)
If 
$\frac {1}{\tau^2} < p < \frac 12$, 
$\mu$
has singular continuous component. 
\end{theorem}
\begin{proof}
(1)
follows from 
Theorem \ref{taiou}. 
To prove 
(2), 
we use the following fact \cite{Rogers} : 
set 
\[
D_{\mu}(x)
:=
\limsup_{ \delta \downarrow 0}
\frac {\mu(x - \delta, x + \delta)}{\delta}, 
\quad
A :=
\{ x | D_{\mu}(x) = \infty \}.
\]
Then 
$1_{A} d \mu$
is singular w.r.t. the Lebesgue measure. 
Take any 
$(O'_1, O'_2, \cdots, O'_n) \in \{ R, L \}^n$
and let 
$k_n = \sharp \{ 1\le j \le n \; | \; O'_j = R\}$, 
$l_n = \sharp \{ 1\le j \le n \; | \; O'_j = L \}$, 
$k_n + l_n = n$.
Then
$I_n 
= I_n (O'_1, O'_2, \cdots, O'_n) 
:= 
(\Phi \circ \Psi)^{-1} (
\{ w = (O_1, O_2, \cdots ) \in W \; | \; 
O_1 = O'_1, O_2 = O'_2, \cdots, O_n = O'_n \})$
satisfies
\[
| I_n | =
\left(
\frac {1}{\tau}
\right)^{k_n}
\left(
\frac {1}{\tau^2}
\right)^{l_n}
=
\frac {1}{\tau^{k_n + 2l_n}}, 
\quad
\mu(I_n) = p^{k_n} q^{l_n}
\]
so that we have
\begin{eqnarray*}
r_n :=
\frac {\mu (I_n) }{| I_n |}
&=&
\frac {p^{k_n} q^{l_n}}
{
\left( \frac {1}{\tau} \right)^{ k_n+2l_n }
}
=
( p \tau )^n 
\left(
\frac {(1-p) \tau}{p}
\right)^{l_n}.
\end{eqnarray*}
Define 
$x$
and 
$\alpha$
by 
\[
p \tau  =: x < 1, 
\quad
\frac {\tau}{x}(\tau - x)
=
x^{- \alpha}, 
\]
Then we have 
$\alpha > 1$
and 
\begin{eqnarray}
r_n
=
x^n 
x^{- \alpha l_n }
=
\left(
\frac 1x
\right)^{\alpha l_n - n}
\label{ratio}
\end{eqnarray}
For 
$w = (O_1, O_2, \cdots ) \in W$
let 
$k_n (w)= \sharp \{ 1\le j \le n \; | \; O_j = R\}$, 
$l_n (w)= \sharp \{ 1\le j \le n \; | \; O_j = L\}$, 
$k_n(w) + l_n(w) = n$.
By 
(\ref{ratio})
$\mu |_A$
is singular continuous, where
\[
A := (\Phi \circ \Psi)^{-1}( 
\{ w \in W \; | \; 
l_n(w) > \frac {n}{\alpha}
\mbox{ for infinitely many
$n$
 } \}).
\]
Lemma \ref{one} below shows 
$\mu (A)> 0$.
\QED
\end{proof}
Let 
\[
B = (\Phi \circ \Psi)^{-1}(
\{ w \in W \; | \;
l_n (w) \le \frac {n}{\alpha}
\mbox{ for infinitely many
$n$
 } \})
\]
so that 
${\bf T} =A \cup B = A \cup (B \setminus A)$.
\begin{lemma}
\label{one}
If 
$\frac {1}{\tau^2} < p < \frac 12$, 
$\mu (A) >0$. 
\end{lemma}
\begin{proof}
Suppose 
$\mu (A) =0$,  
then 
$\mu (B \setminus A) > 0$. 
By definition, 
\begin{eqnarray*}
A \setminus B
&=&
(\Phi \circ \Psi)^{-1}(
\{ w \in W \; | \; n \gg 1, \; l_n(w) > \frac {n}{\alpha} \}
)
\\
B \setminus A
&=&
\bigcup_{N \ge 1}
\bigcap_{n \ge N}
(\Phi \circ \Psi)^{-1}(
\{ x \; | \; \; k_n (w) \ge (1 - \frac {1}{\alpha}) n \}
)
=:
\bigcup_{N \ge 1} (B \setminus A)_N.
\end{eqnarray*}
Since
$(B \setminus A)_N$
is monotone increasing, 
$\mu ((B \setminus A)_N) > 0$
for some 
$N$. 
Let 
$(B \setminus A)'_N$
be the set with 
$R$
and 
$L$
being exchanged in 
$(B \setminus A)_N$ : 
\[
(B \setminus A)'_N
= 
\bigcap_{n \ge N}
(\Phi \circ \Psi)^{-1}(
\{ w \in W \; | \; 
l_n(w) \ge (1 - \frac {1}{\alpha}) n \}
).
\]
Since 
$\frac {1}{\tau^2} < p < \frac {1}{\tau}$, 
$l_n (w) \ge (1 - \frac {1}{\alpha}) n$
implies 
$l_n (w) > \frac {n}{\alpha}$
so that 
$(B \setminus A)'_N \subset A \setminus B$.
Hence
\[
\mu ((B \setminus A)'_N) 
\le
\mu (A \setminus B) (=0).
\]
It suffices to show
\begin{equation}
(0 <)
\;
\mu ((B \setminus A)_N)
\le
\mu ((B \setminus A)'_N)
\label{flat}
\end{equation}
which leads us to a  contradiction. 
To see (\ref{flat}), note that we may assume 
\[
l_{N-1}((\Phi \circ \Psi)(x)) \le \frac {N-1}{2}
\]
for 
$x \in (B \setminus A)_N$ 
by letting
$N$
large if necessary. 
Hence
if we exchange 
$R$
with 
$L$, 
$\sharp L$
increases in 
$( B\setminus A)_N$. 
Since 
$
\mu (I_n)
=
\left(
\frac {x}{\tau}
\right)^n 
(\tau x^{\alpha})^{-l_n}
$
and since 
$\tau x^{\alpha} < 1$
for 
$p < \frac 12$, 
$\mu(I_n)$
is monotone increasing w.r.t. 
$l_n$
which implies
$(\ref{flat})$.
\QED
\end{proof}
%
\section{Some combinatorial aspects of admissible words}
In this section, 
we consider general circle map sequences except in subsection 3.2, and 
use the symbol 
$A$, $B$
instead of 
$1$, $0$
respectively. 
Let 
$P_n$
be the set of admissible words(factors of 
$v_0$) of length 
$n$. 
$| P_n | = n+1$
is well known. 
We can find 
$t_n \in P_n$
uniquely such that 
$t_n A$, $t_n B \in P_{n+1}$
and for 
$a \in P_n \setminus \{ t_n \}$
there exists unique
$C = C(a) \in \{ A, B\}$
with 
$a C(a) \in P_{n+1}$
\footnote{This is called the right special factor \cite{L}}. 
For any 
$k$
with 
$n \le q_k -2$
we have 
$t_n
=
(\pi_k(n), \pi_k(n-1), \cdots, \pi_k(1))$. 
In this section
we study some combinatorial properties of admissible words. 
\subsection{Exhausting point}
For 
$n \ge 2$, 
let
$f(n) \in {\bf N}$
be the smallest number 
where we have seen all words in 
$P_n$
in 
$( v_0 (n) )_{ n \ge 1}$.
For instance in the Fibonacci word, 
\[
ABA
\underline{A}_2
BA
\underline{B}_3 \underline{A}_4
ABA
\underline{A}_5 \underline{B}_6
AB \cdots
\]
$\underline{ * }_n$
corresponds to 
$f(n)$. 
Hence 
$f(2) = 4$, $f(3) = 7$, $f(4)=8$
in this case.
\begin{theorem}
\label{f}
\quad\\
Let 
$n \ge 2$
and take
$k = 0, 1, \cdots$
such that
\[
q_k \le n \le q_{k+1}-1.
\]
Then writing 
$n = q_k + j$, 
we have
\begin{equation}
f(q_k + j)
=
q_{k+1} + q_k - 1 + j, 
\quad
j = 0, 1, \cdots, q_{k+1} - q_k -1.
\label{exhausting point}
\end{equation}
\end{theorem}
The corresponding exhausting points lies from the letter  next to 
$s_{k+1} \pi_k$
to the last letter in 
$s_{k+1} \pi_{k+1}$.
Therefore 
$(v_0(f(n)))_{n \ge 2}$
coincides with the original circle map sequence 
$(v_0 (n) )_{n \ge 1}$. 
\begin{corollary}
\[
v_0 (f(n+1)) = v_0 (n), 
\quad
n =1,  2, 3, \cdots
\]
\end{corollary}
Let 
$g(n) \in {\bf N}$
be the smallest number where we have seen both 
$t_{n-1} A$
and 
$t_{n-1} B$. 
As a preparation, 
we prove the following lemma.
\begin{lemma}
\label{g}
$f(n)$
is the smallest number satisfying following conditions. 
\[
(i)
\;
f(n-1) + 1 \le f(n),
\quad
(ii)
\;
g(n) \le f(n)
\]
\end{lemma}
\begin{proof}
We first show that 
$f(n)$
satisfies (i), (ii). 
(ii)
is clear.
We should have 
$f(n-1) < f(n)$, 
because cutting the rightmost letter in words in 
$P_n$
yields all words in 
$P_{n-1}$. 
Hence 
$f(n)$
satisfies (i). 
Thus 
it suffices to show that if a number
$f(n)$
satisfies (i), (ii), then we have already seen all words in
$P_n$
at
$f(n)$. 
By the equation 
$P_n
=
\{ a C(a) \}_{a \in P_{n-1} \setminus \{ t_{n-1}\}}
\cup
\{ t_{n-1}A, t_{n-1}B \}$, 
this is clear. 
\QED
\end{proof}
{\it Proof of Theorem \ref{f}}  
$\;$
We prove
(\ref{exhausting point})
by induction on 
$k$.
Let
\[
k_0 = \cases{
2 & ($a_1=1$, $a_2=1$) \cr
1 & ($a_1=1$, $a_2 \ge 2$, or $a_1 = 2$) \cr
0 & ($a_1 \ge 3$) \cr
}
\]
so that
$q_{k_0} \le 2 \le q_{k_0+1}-1$.
For
$n = 2, 3, \cdots, q_{k_0+1}-1$, 
it is straightforward to see 
(\ref{exhausting point}).
We next suppose that 
(\ref{exhausting point}) 
holds true for 
$k_0, k_0+1, \cdots, k-1$
and would like to prove it for $k ( \ge k_0 + 1)$. 
Let
$q_k \le n \le q_{k+1}-1$.
We note that 
$t_{n-1}$
is a subword of 
$\pi_{k+1}$
and is not a subword of 
$\pi_k$.
Since
$s_{k+1} \pi_k = s_k \pi_{k+1}$, 
both
$\pi_{k+1} AB$
and
$\pi_{k+1} BA$
are subwords of 
$s_{k+1} s_k$, 
which implies
\begin{equation}
g(n) \le
q_{k+1} + q_k - 1.
\label{upper bound}
\end{equation}
On the other hand, 
since we suppose 
(\ref{exhausting point}) 
for 
$k-1$, 
$f(q_k-1) = 2 q_k - 2$. 
In 
$( v_0(1), v_0(2), \cdots, v_0(2 q_k-2))$, 
we have 
$(2 q_k - 2) - (q_k-1) + 1= q_k$
of words of length 
$q_k-1$.
Since 
$| P_{q_k-1} | = q_k$, 
we find each element of 
$P_{q_k-1}$ 
only once in 
$( v_0(1), v_0(2), \cdots, v_0(2 q_k-2))$.
Hence
$t_{q_k-1}$
appears only once in 
$s_k \pi_k$. 
In what follows, 
we suppose that
$k$
is even.
For 
$k$ odd, we have only to exchange 
$AB$
with 
$BA$
in the argument below. 
Since
$t_{q_k-1}$
is the last subword of length 
$q_k-1$
in 
$s_k \pi_{k-1}$, 
it appears 
$a_{k+1}$
times in 
$s_{k+1}$
and they have the form of 
$t_{q_k-1} BA$. 
The other one
$t_{q_k-1} AB$
appears as the last subword of 
$s_{k+1} s_k$.
Since
$t_{q_k-1}$
appears only once in 
$s_k \pi_k$, 
they exhaust all 
$t_{q_k-1}$'s
in 
$s_{k+1} s_k$. 
Therefore we have
$g(q_k)
=
q_{k+1} + q_k -1$. 
Since 
$f(q_k-1) < g(q_k)$, 
\[
f(q_k) = q_{k+1} + q_k -1, 
\]
by Lemma \ref{g}. 
By the monotonicity of 
$g$, 
we have 
$
g(n) \ge q_{k+1} + q_k -1
$
for
$q_k \le n \le q_{k+1}-1$
and together with
(\ref{upper bound}), 
$
g(n) = q_{k+1} + q_k -1
$
for such 
$n$. 
By Lemma \ref{g} again,
\[ 
f(q_k + j) = g(q_k + j) + j
=
q_{k+1} + q_k - 1 + j
\]
for 
$j = 0, 1, \cdots, q_{k+1} - q_k - 1$
which proves
(\ref{exhausting point})
for 
$k$. 
\QED
%
\subsection{The classification of 
$P_n$
and frequency : Fibonacci case}
We consider 
the classification of words in 
$P_n$
in terms of their frequency. 
We study
Fibonacci case in this subsection. 
Let 
$\{F(n) \}$
be the Fibonacci sequence defined by
\[
F_1 = 1, \quad F_2 = 2, 
\quad
F_{n+1} = F_n + F_{n-1}.
\]
By (\ref{pq}), 
$F_n =q_n = | s_n |$.
\begin{theorem}
\label{Fibonacci case}
Let 
$a_n = 1$.
We can decompose 
$P_n$
into three disjoint subsets 
\[
P_n = A_n \cup B_n \cup C_n
\]
which are given explicitly as follows. 
Let 
$n = F_k + j$, 
$j = 0, 1, \cdots, F_{k-1}-1$.
\begin{eqnarray*}
A_n : \;
&&
( v_0(1), \cdots, v_0(F_k+j) )
\\
&&\qquad\qquad\ddots\\
&&
\qquad
( v_0(1+m), \cdots, v_0(F_k+j+m) )
\quad
(m=0, 1, \cdots, F_{k-1}-j-2)
\\
&&\qquad\qquad\qquad\qquad\ddots\\
&&
\qquad\qquad
( v_0(F_{k-1}-j-1), \cdots, v_0(F_{k+1}-2) )
\\
&& \\
B_n : \;
&&
( v_0(F_{k-1}-j), \cdots, v_0(F_{k+1}-1) )
\\
&&\qquad\qquad\ddots\\
&&\qquad
( v_0(F_{k-1}-j+m), \cdots, v_0(F_{k+1}-1+m) )
\quad
(m=0, 1, \cdots, F_{k-2}+j)
\\
&&\qquad\qquad\qquad\qquad\ddots\\
&&\qquad\qquad
( v_0(F_k), \cdots, v_0(2F_k + j -1) )
\\
&& \\
A_n : \;
&&
( v_0 (F_k+1), \cdots, v_0(2 F_k + j) )
\\
&&\qquad\qquad\ddots\\
&&\qquad
( v_0 (F_k+1+m), \cdots, v_0(2 F_k + j + m) )
\quad
(m=0, 1, \cdots, F_{k-1}-j-2)
\\
&&\qquad\qquad\qquad\qquad\ddots\\
&&\qquad\qquad
( v_0 (F_{k+1}-j-1), \cdots, v_0(F_{k+2}-2) )
\\
&& \\
C_n : \;
&&
( v_0(F_{k+1}-j), \cdots, v_0(F_{k+2}-1) )
\\
&&\qquad\qquad\ddots\\
&&\qquad
( v_0(F_{k+1}-j+m), \cdots, v_0(F_{k+2}-1+m) )
\quad
(m=0, 1, \cdots, j)
\\
&&\qquad\qquad\qquad\qquad\ddots\\
&&\qquad\qquad
( v_0(F_{k+1}), \cdots, v_0(F_{k+2}+j-1) )
\end{eqnarray*}
They are characterized as follows. 
$A_n$
consists of all words in 
$P_n$
contained in 
$s_k \pi_{k-1} (= \pi_{k+1})$
and each ones in
$A_n$
appear twice before arriving at 
$f(n)$, while those in 
$B_n$, $C_n$
appear only once. 
Starting from
$v_0(1)$, 
we see words in 
$A_n$
one after another. 
Words in 
$B_n$
begin to appear after we have seen all words in 
$A_n$. 
Words in 
$A_n$
appear for the second time after we have seen words in 
$B_n$. 
Words in 
$C_n$
appear after we have seen all words in 
$A_n$
for the second time. 
Moreover\\
\noindent
(1)
$| A_n | =(F_{k-1} - j -1)$
and each word in 
$A_n$
has overlaps of length
$j$
or 
$(F_{k-2} + j)$
with itself. \\
(2)
$| B_n | = ( F_{k-2} + j + 1 )$
and each word in 
$B_n$
has overlaps of length 
$j$
or has distance 
$(F_{k-1}-j)$
with itself. \\
(3)
$| C_n | = (j+1)$
and each word in 
$C_n$
has distance 
$(F_{k-1}-j)$ 
or 
$(F_{k+1}-j)$
with itself. 
\end{theorem}
Since each words in 
$A_n$
has overlaps with itself, it can cover 
$v \in \Omega$
if overlap is allowed. 
Penrose tiling
has analogous property
\cite{Gummelt, Komatsu-Nakano}.
\begin{proof}
Before arriving at
$f(n) = f(F_k + j)= F_{k+2}+j-1$, 
we see 
$F_{k+1}$
words of length 
$n$. 
Since 
$| P_n | =(F_k + j + 1)$, 
at least 
$F_{k+1} - (F_k + j + 1) = F_{k-1} - j - 1$
words should appear more than twice. 
On the other hand since 
$
s_{k-1}  \pi_k = \pi_{k+1}, 
$
the words of length 
$n$
contained in 
$\pi_{k+1}$
(there are 
$F_{k-1} - j-1$
of these)
should appear at least twice before arriving at 
$f(n)$. 
Let 
$A_n$
be the set of such words. 
Then 
the words in 
$P_n \setminus A_n$
appears only once. 
The properties of 
$B_n$, $C_n$ 
follows from looking at words in 
$(v_0 (1), \cdots, v_0(f(n)))$
explicitly, and the length of 
overlaps and distance follows from the 
$(k-1, k)$-partition of 
$v_0$. 
\QED
\end{proof}
We 
next compute the frequency of words in 
$P_n$. 
\begin{theorem}
\label{frequency}
Let
$v \in \Omega$
and let 
$n = F_k + j$, $j = 0, 1, \cdots, F_{k-1}-1$.
For
$a \in P_n$,
\[
\lim_{N \to \infty}
\frac { \sharp 
\mbox{ 
$a$'s in 
}
(v (1), v (2), \cdots, v(N))
}
{N}
=
\cases{
\frac {1}{\tau^{k-1}} 
&
$(a \in A_n)$ \cr
\frac {1}{\tau^{k}} 
&
$(a \in B_n)$ \cr
\frac {1}{\tau^{k+1}}
&
$(a \in C_n)$\cr}
\]
\end{theorem}
\begin{proof}
Due to the unique ergodicity of the dynamical system
$(\Omega, T)$
($(Tv)(n):=v(n+1)$, 
$v\in \Omega$
is the shift operator), 
we can work on some subsequence. 
In the 
$(k-1, k)$-partition, 
the frequency of 
$s_k$, $s_{k-1}$
are equal to 
$\alpha$, 
$1-\alpha$
respectively.
Thus when
$| s_k | + | s_{k-1} | = N$, 
$\sharp \{ s_k \} = N \alpha (1 + o(1))$, 
$\sharp \{ s_{k-1} \} = N (1- \alpha) (1 + o(1))$
so that the number of letters is equal to 
$\left\{
\alpha F_k  + (1 - \alpha ) F_{k-1} 
\right\}N
(1+o(1))$. 
Each word in 
$A_n$
is found in every 
$s_k, s_{k-1}$
while each word in 
$B_n$ (resp. $C_n$)
is found in every 
$s_k$ (resp. $s_{k-1}$).
Hence
\begin{eqnarray*}
r_A &=&
\frac {N}
{(\alpha F_k  + (1 - \alpha ) F_{k-1} )N}
=
\frac {1}{\tau^{k-1}}
\\
r_B &=&
\frac {N\alpha}
{(\alpha F_k  + (1 - \alpha ) F_{k-1} )N}
=
\frac {1}{\tau^{k}}
\\
r_C
&=&
\frac {N(1-\alpha)}
{(\alpha F_k  + (1 - \alpha ) F_{k-1} )N}
=
\frac {1}{\tau^{k+1}}.
\end{eqnarray*}
\QED
\end{proof}
\begin{theorem}
\label{three-distance}
Let
$n = F_k + j$, 
$j = 0, 1, \cdots, F_{k-1}-1$
and 
$w \in P_n$. 
Then
$I (w) := \Psi^{-1}
(\{ v \in \Omega \; | \; 
( v(0), v(1), \cdots, v(n-1) ) = w
\})$
satisfies
\[
| I(w) | = \cases{
\frac {1}{ \tau^{k-1} } & ($w \in A_n$) \cr
\frac {1}{ \tau^{k} } & ($w \in B_n$) \cr
\frac {1}{ \tau^{k+1} } & ($w \in C_n$) \cr
}
\]
In other words, 
the width of intervals in 
${\bf T}$
corresponding to words in 
$A_n, B_n$ 
and
$C_n$
are
$\frac {1}{\tau^{k-1}}$, 
$\frac {1}{\tau^{k}}$
and 
$\frac {1}{\tau^{k+1}}$
respectively.
\end{theorem}
Since the endpoints 
in these intervals are equal to the set 
$\{ x \; | \; x \equiv - j  \alpha \pmod 1,  \; j=1, 2, \cdots, n \}$, 
Theorem \ref{three-distance}
implies the three-distance theorem \cite{Sos, AB}. 
\begin{proof}
This follows directly from 
Theorem \ref{frequency}
and the ergodic theorem. 
It is also possible 
to prove Theorem \ref{three-distance} directly 
by using Theorem \ref{taiou} and inductive argument. 
In doing so, 
we note that 
$t_n \in A_n$
for 
$n = F_k + j$, 
$j = 0, 1, \cdots, F_{k-1}-2$
(resp. 
$t_n \in B_n$
for 
$n = F_k + F_{k-1}-1= F_{k+1}-1$)
and 
$t_n A$, $t_n B$
belong to 
$B_{n+1}$, $C_{n+1}$, 
while 
the corresponding intervals 
are divided into two intervals with ratio 
$\tau : 1$.
\QED
\end{proof}
\begin{remark}
Let us call
$w_n \in P_n$
exhausting word if the endpoint of which is located at   
$f(n)$. 
By Theorem \ref{Fibonacci case}, 
$w_{F_k} = A \pi_k A$
(resp. 
$w_{F_k} = B \pi_k B$)
if 
$k$ 
is even (resp. odd). 
Then by the fact that 
$t_n \notin C_n$
and the embedding procedure we see that 
$\Psi^{-1}(w_n)$
is the interval which is closest to the endpoint of 
${\bf T}$ : if
$n=F_k +j$, $j=0, 1, \cdots, F_{k-1}-1$, 
\[
\Psi^{-1}(w_n) = \cases{
\left[
1 - \frac {1}{\tau^{k+1}}
\right)
&
($k$ : even) \cr
\left[ 
0, \frac {1}{\tau^{k+1}}
\right)
&
($k$ : odd) \cr
}
\]
\end{remark}
\subsection{Classification of 
admissible words and frequency : general case}
The results in previous subsection 
is directly extended to the general case, 
though the statement becomes slightly complicated. 
We only state the results.
For given 
$n$, 
take 
$k$
such that
$q_k \le n \le q_{k+1}-1$. \\

\noindent
(1)Classification\\
(i)
$q_k \le n \le q_k + q_{k-1} -1$ : writing 
$n = q_k + j$, 
$j = 0, 1, \cdots, q_{k-1}-1$, 
we have
\begin{eqnarray*}
A_n : \;
&&\left(
v_0(1+m + p q_k), \cdots, v_0(s_k + j + m + p q_k)
\right)
\\
&&m = 0, 1, \cdots,q_{k-1}-j-2, 
\quad
p = 0, 1, \cdots, a_{k+1}
\\
B_n  : \; 
&&\left(
v_0(q_{k-1}-j+m + p q_k), \cdots, v_0(q_k + q_{k-1}-1+m + p q_k)
\right)
\\
&&
m=0, 1, \cdots, q_k - q_{k-1} + j, 
\quad
p = 0, 1, \cdots, a_{k+1}-1
\\
C_n : \;
&&\left(
v_0(q_{k+1}-j+m), \cdots, v_0(q_{k+1}+q_k-1+m)
\right)
\\
&&
m=0, 1, \cdots, j.
\end{eqnarray*}
The order of their appearance is 
$\overbrace{
(A_n, B_n), \cdots, (A_n, B_n)}^{a_{k+1}}, (A_n, C_n)$. 
They are characterized as follows. 
\\
$A_n$ : 
they are the words of length 
$n$
in 
$s_k \pi_{k-1}$, 
$| A_n | = q_{k-1}-1-j$, 
and appear
$(a_{k+1}+1)$-times 
before arriving at 
$f(n)$. 
Each one has overlap of length 
$j $, $(q_k - q_{k-1} + j)$
with itself. 
\\
$B_n$ : 
$| B_n |
=
q_k - | A_n |
=
q_k - q_{k-1} + 1 + j$
and appear 
$a_{k+1}$
before arriving at 
$f(n)$.
Each one has overlap of length
$j$ 
or has distance of length 
$(q_{k-1} - j)$. \\
$C_n$ : 
$| C_n |=j+1$
and appear only once before arriving at 
$f(n)$. 
Each one has distance of length 
$(q_{k+1}-j)$, $(q_{k+1}-q_k-j)$
from itself. \\

\noindent
(ii)
$l q_k + q_{k-1} 
\le n \le 
(l+1) q_k + q_{k-1} -1$, 
$l = 1, 2, \cdots, (a_{k+1}-1)$ : 
writing 
$n = l q_k + q_{k-1} + j$, 
$j = 0, 1, \cdots, (q_k-1)$, 
we have 
\begin{eqnarray*}
A_n : \; 
&&\left(
v_0(1+m + p q_k), \cdots, v_0(l q_k + q_{k-1} + j + m + p q_k)
\right)
\\
&& m = 0, 1, \cdots, q_k - j - 2, 
\quad
p = 0, 1, \cdots, a_{k+1}-l
\\
B_n : \; 
&& \left(
v_0(q_k-j+m + p q_k), \cdots, v_0((l+1) q_k + q_{k-1} - 1 + m + p q_k)
\right)
\\
&& m = 0, 1, \cdots, j, 
\quad
p = 0, 1, \cdots, a_{k+1}-l-1
\\
C_n : \; 
&&\left(
v_0(q_{k+1}-(l-1) q_k - q_{k-1} - j+m), 
\cdots, 
v_0(q_{k+1}+q_k-1+m)
\right)
\\
&&
m = 0, 1, \cdots, (l-1) q_k + q_{k-1} + j.
\end{eqnarray*}

\noindent
The order of their appearance is 
$\overbrace{
(A_n, B_n), \cdots, (A_n, B_n)}^{a_{k+1}-l}, (A_n, C_n)$. 
They are characterized as follows. \\
$A_n$ : 
they are the words of length 
$n$
in 
$s_k^{l+1} \pi_{k-1}$, 
$| A_n |
=
q_k - 1 - j$, 
and appear 
$(a_{k+1}-l+1)$-times 
before arriving at 
$f(n)$.
Each one has overlaps with itself of length
$(l-1) q_k + q_{k-1} + j$, 
$(l-2) q_k + q_{k-1} + j$, 
$\cdots$
$[(2l - a_{k+1}-1) ]_+ q_k + q_{k-1} + j$ 
and 
$j$.\\
\noindent
$B_n$ : 
$| B_n |
=j+1$, 
and appear 
$(a_{k+1} - l)$-times 
before arriving at 
$f(n)$. 
Each one has overlaps with itself of length
$(l-1) q_k + q_{k-1} + j$, 
$(l-2) q_k + q_{k-1} + j$, 
$\cdots$
$[(2l - a_{k+1}) ]_+ q_k + q_{k-1} + j$ 
or has distance of length
$q_k-j$. 
\\
\noindent
$C_n$ : 
$| C_n |
=
(l-1) q_k + q_{k-1} + 1 + j$, 
and appear only once before arriving at 
$f(n)$.
Each one has distance from itself of length
$(a_{k+1} - l) q_k - j$, 
$(a_{k+1}+1- l) q_k - j$.
\\

\noindent
(2)Frequency\\
To compute the frequency, set 
\[
\beta_k := 
[1, a_{k+2}, a_{k+3}, \cdots].
\]
Let
$r_{A_n}$
(resp. $r_{B_n}$, $r_{C_n}$)
be the frequency of the words in 
$A_n$
(resp. $B_n$, $C_n$). 
(i)
$q_k \le n \le q_k + q_{k-1} - 1$ : 
\begin{eqnarray*}
r_{A_n}
&=&
\frac {
\beta_k(a_{k+1}+1)+(1-\beta_k)}
{\beta_k q_{k+1} + (1 - \beta_k) q_k}
\\
r_{B_n}
&=&
\frac {
\beta_k a_{k+1}+(1-\beta_k)}
{\beta_k q_{k+1} + (1 - \beta_k) q_k}
\\
r_{C_n}
&=&
\frac {
\beta_k}
{\beta_k q_{k+1} + (1 - \beta_k) q_k}.
\end{eqnarray*}
(ii)
$l q_k + q_{k-1} \le n \le 
(l+1) q_k + q_{k-1} - 1$, 
$l=1, 2, \cdots, (a_{k+1}-1)$ : 
\begin{eqnarray*}
r_{A_n}
&=&
\frac {
\beta_k (a_{k+1}-l+1) + (1 - \beta_k)}
{\beta_k q_{k+1} + (1 - \beta_k) q_k}
\\
r_{B_n}
&=&
\frac {
\beta_k (a_{k+1}-l) + (1 - \beta_k)}
{\beta_k q_{k+1} + (1 - \beta_k) q_k}
\\
r_{C_n}
&=&
\frac {
\beta_k}
{\beta_k q_{k+1} + (1 - \beta_k) q_k}.
\end{eqnarray*}
%
\section{Appendix 1 : Basic properties of embedding procedure}
%
\subsection{Fixed point of $\Phi$}
In this subsection, 
we would like to represent the fixed point 
$f = (L,R,R,L,R,L,R,R,L, \cdots) \in W$
of 
$\Phi : \Omega \to W$
in terms of the recursion relation of the sequence of words
$\{ u_n \}_{n=0}^{\infty}$
such that 
$f$
is the right limit of that : 
$f = \lim_{n \to \infty} u_n$.
In whalt follows, we identify
$1 \leftrightarrow R$, 
$0 \leftrightarrow L$.
Let 
\begin{eqnarray*}
u_0 &=& s_0
\\
v_0 &=& s_1
\\
k(0)
&=&
1
\end{eqnarray*}
$k(0)$
stands for the suffix of 
$s_{\sharp}$
in 
$v_0$. 
To go further, we prepare some notations. 
For 
$s \in {\cal S}
:= 
\left\{ s_{l_1} s_{l_2} \cdots s_{l_N}  : 
l_1 < l_2 < \cdots < l_N, 
\;
N \in {\bf N}
\right\}$, 
we define an operation 
${\cal O}(v)$
as follows. 
Arrange its elements like
$(R, L, R, R, L, \cdots)$, 
partition it in terms of 
$R$ 
and 
$RL$ 
like
$((RL), R, (RL), \cdots)$, 
and replace 
$R$(resp. $RL$) by $A_R$ (resp. $A_{RL}$). 
\begin{eqnarray*}
v &=& s_{l_1} s_{l_2} \cdots s_{l_r}
\\
&=&
(O_1, O_2, O_3, \cdots, O_N), 
\quad
O_j = R \mbox{ or } RL
\\
&& \qquad \qquad \downarrow 
\\
{\cal O}(v)
&=&
(A_1, A_2, A_3, \cdots, A_N), 
\quad
A_j = {\cal O}(O_j) =  A_R \mbox{ or } A_{RL}
\end{eqnarray*}
where
${\cal O}(R) = A_R$, 
${\cal O}(RL) = A_{RL}$, 
whose operation on 
${\cal S}_0 :=
\{ s_l \, : \, l \in {\bf N} \}$, 
are defined by
\[
A_R s_k := s_{k+1}, 
\quad
A_{RL} s_k := s_{k+3}, 
\]
and the action of 
${\cal O}(v)$
on
$s_k$
is defined by
\begin{eqnarray*}
(A_1, A_2, A_3, \cdots, A_N) s_k
&:=&
(A_1 s_k) (A_2 A_1 s_k) \cdots 
(A_N A_{N-1} \cdots A_2 A_1 s_k).
\end{eqnarray*}
By using notations above, 
the recursion relation between
$(u_n, v_n, k(n))$
and
$(u_{n+1}, v_{n+1}, k(n+1))$
is given by
\begin{eqnarray*}
u_{n+1} &=& u_n v_n
\\
v_{n+1} &=& {\cal O}(v_n) s_{k(n)}
\\
k(n+1)
&=&
k(n) +  (3 \sharp \{\mbox{$(RL)$'s in $v_n$}\} + \sharp \{\mbox{$R$'s in $v_n$}\} ).
\end{eqnarray*}
$k(n)$
is equal to the suffix of the rightmost word in 
$v_n \in {\cal S}$. 
The followings 
are computations of a few of them. 
\begin{eqnarray*}
&&\cases
{u_1 = u_0 v_0 = s_0 s_1 \cr
%
v_1 = {\cal O}(v_0)s_{k(0)} = {\cal O}(s_1) s_1 = s_2
\cr
%
k(1) 
=
k(0) + (3 \sharp \{\mbox{$(RL)$'s in $v_0$}\} + \sharp \{\mbox{$R$'s in $v_0$}\} )
= 1 + 1 = 2,\cr
}\\
%
&&\cases{
u_2 = u_1 v_1 = u_1 s_2 \cr
%
v_2 = {\cal O}(v_1) s_{k(1)}
=
{\cal O}(s_2) s_2
=
{\cal O}(RL) s_2 = s_5
\cr
k(2) = k(1) 
+ (3 \sharp \{\mbox{$(RL)$'s in $v_1$}\} + \sharp \{\mbox{$R$'s in $v_1$}\} )
=5,}
\\
&&\cases{
u_3 = u_2 v_2 = u_2 s_5
\cr
v_3 = {\cal O}(v_2) s_{k(2)}
= {\cal O}(s_5) s_5
=
{\cal O}((RL) R (RL) (RL) R) s_5
\cr
\quad
=
(A_{RL}, A_R, A_{RL}, A_{RL}, A_R) s_5
\cr
\quad=
(A_{RL} s_5)
(A_R A_{RL} s_5)
(A_{RL} A_R A_{RL}  s_5)
\cr
\qquad\qquad
(A_{RL} A_{RL} A_R A_{RL} s_5)
(A_R A_{RL} A_{RL} A_R A_{RL} s_5)
\cr
\quad=
s_8 s_9 s_{12} s_{15} s_{16}
\cr
k(3) =
k(2)
+ (3 \sharp \{\mbox{$(RL)$'s in $v_2$}\} + \sharp \{\mbox{$R$'s in $v_2$}\} )
= 16
}
\end{eqnarray*}
%

\subsection{Concrete examples}
From the discussion in subsection 2.2, 
$\Phi (v)$
is seen to reflect some combinatorial aspects of 
$v \in \Omega$, 
and 
$\Phi (v)$
in turn can be derived 
by Theorem \ref{taiou}. 
In this subsection we explicitly give 
$\Phi (v)$
for some examples of
$v$ : 
$v_0(\cdot - m)$, $v'_0(\cdot - m)$ 
and 
$v_{AA}$, $v_A$, $v_B$
defined later. \\

\noindent
(1)
$v_0(\cdot - m)$, $v'_0(\cdot - m)$ : 
it is easy to see
\begin{eqnarray}
\Phi (v_0) =(L, L, \cdots), 
\quad
\Phi (v'_0) = (R, L, L, \cdots).
\label{v0}
\end{eqnarray}
Moreover
\begin{eqnarray*}
\Omega_L
:=
\{ v_0 (\cdot - m), \; v'_0 (\cdot - m) 
\; | \; 
m \ge 0 \}, 
\quad
\Omega_R
:=
\{ v_0 (\cdot + m), \; v'_0 (\cdot + m) 
\; | \; 
m \ge 1 \}
\end{eqnarray*}
satisfy
\begin{eqnarray}
\Phi (\Omega_L )
&=&
\{ (O_1, O_2, \cdots) 
\; | \; 
O_j = L
\mbox{ for large } j \}
\label{L}
\\
\Phi (\Omega_R )
&=&
\{ (O_1, O_2, \cdots) 
\; | \; 
O_j = R
\mbox{ for large } j \}
\label{R}
\end{eqnarray}
In fact, to see (\ref{L})
we note that 
$\Psi^{-1}(v_0 (\cdot - m)) \in D_-$ ($m \ge 0$)
by definition. 
Therefore, by a successive application of 
$R$ or $L$, 
say after the $k$-th step
we reach the interval with 
$\Psi^{-1}(v_0 (\cdot - m))$
its left endpoint, and then we set 
$O_{k+1} = R$, 
$O_{k+2} = O_{k+3} = \cdots = L$.
For 
$\Phi (v'_0 (\cdot - m))$, 
we approach 
$\Psi^{-1}(v_0 (\cdot - m))$
from the opposite direction. 
Conversely, if 
$w \in 
\{ (O_1, O_2, \cdots) 
\; | \; 
O_j = L
\mbox{ for large } j \}$, 
we have 
$(\Phi \circ \Psi)^{-1}(w) \in D_-$
by Theorem \ref{taiou}. 

To see (\ref{R}), 
we recall that 
$m \in {\bf N}$
has the following unique representation
\[
m = F_{k_1} + F_{k_2} + \cdots+F_{k_N}, 
\quad
l_j := k_j - k_{j-1} \ge 2, 
\;
j = 2, 3, \cdots, N, 
\]
by which 
$\Phi(v_0(\cdot + m))$, $\Phi(v'_0(\cdot + m))$
are given explicitly below.\\
(i)
$k_1$ : odd
\begin{eqnarray*}
&&
\Phi(v_0(\cdot + m))
=
\Phi(v'_0(\cdot + m))
\\
&=&
(R, 
\overbrace{L, \cdots, L}^{\frac {k_1-1}{2}}, 
\overbrace{R, \cdots, R}^{l_2-1}, L, 
\overbrace{R, \cdots, R}^{l_3-2}, L, 
\cdots,  
\overbrace{R, \cdots, R}^{l_N-2}, L,  
R, R, \cdots
)
\end{eqnarray*}
(ii)
$k_1$ : even
\begin{eqnarray*}
&&
\Phi(v_0(\cdot + m))
=
\Phi(v'_0(\cdot + m))
\\
&=&
(
\overbrace{L, \cdots, L}^{\frac {k_1}{2}}, 
\overbrace{R, \cdots, R}^{l_2-1}, L, 
\overbrace{R, \cdots, R}^{l_3-2}, L, 
\cdots,  
\overbrace{R, \cdots, R}^{l_N-2}, L,  
R, R, \cdots
)
\end{eqnarray*}
The converse is clear. 

For the general case, given 
$\theta \in {\bf T}$
we take a sequence
$\{ N_k \}_{k = 1}^{\infty}$
with 
$N_k \alpha \downarrow \theta$
in
${\bf T}$, 
and then the above argument tells us how to obtain 
$\Phi (v_{\theta})$. \\

\noindent
(2)
symmetric sequences : 
$\Omega$
contains words with mirror symmetry 
\begin{eqnarray*}
v_{AA}
&:=&
\cdots
110101 | 101011
\cdots
=:
h_{AA}^{-1}h_{AA}
\\
v_A
&:=&
\cdots
10110 \underline{1} 01101
\cdots
=: h_A^{-1} A h_A
\\
v_B
&:=&
\cdots
1011 \underline{0} 1101
\cdots
=: h_B^{-1} B h_B
\end{eqnarray*}
which do not belong to 
$\Omega_R \cup \Omega_L$.\\
(i)
$v_{AA}$
 : setting 
$v_{AA}(-1) = v_{AA}(0) = 1$
gives  
$\theta_{AA} := \Psi^{-1}(v_{AA}) = \frac 12$
and
\begin{eqnarray*}
\Phi (v_{AA})
&=&
(R, R, L, R, L, \cdots).
\end{eqnarray*}
(ii)
$v_A$ : 
setting 
$v_A (-1) = 1$
gives 
$\theta_A  := \Psi^{-1}(v_{A})= \frac {\alpha}{2}$
and 
\begin{eqnarray*}
\Phi(v_A)
=
(L, R, L, R, \cdots).
\end{eqnarray*}
(iii)
$v_B$ : 
setting 
$v_B(1) = 0$
gives 
$\theta_B  := \Psi^{-1}(v_{B})= \frac 12 - \frac {3}{2}\alpha$
and 
\begin{eqnarray*}
\Phi (v_B)
=
(R, L, R, L, \cdots).
\end{eqnarray*}
%

\subsection{Symmetric words}
In this subsection, 
we further study some combinatorial properties of 
$v_{AA}$, $v_A$ 
and 
$v_B$.
When 
$n$
is odd, 
$s_{n+3}
=
s_{n+1} s_n s_{n+1}
=
\pi_{n+1} (AB) \pi_n (BA) \pi_{n+1} (AB)$
from which we have
\begin{eqnarray}
\pi_{n+3}
=
\pi_{n+1} (AB) \pi_ n (BA) \pi_{n+1}
\label{asterisque}
\end{eqnarray}
For even 
$n$
we exchange 
$AB$
with 
$BA$. 
Hence 
$\pi_n$
and 
$\pi_{n+3}$
have the same symmetry and 
$v_A$, $v_B$ and $v_{AA}$
can be derived by using this equation for 
$n=3k$, $n=3k+1$ and $n=3k+2$
respectively. 
In fact, define
$h_n$
by the following equation. 
\[
s_n = : 
\cases{
h_n^{-1} \cdot A \cdot h_n, 
&
($n = 3k = 3,6,9, \cdots$)
\cr
h_n^{-1} \cdot B \cdot h_n, 
&
($n = 3k + 1 = 4, 7, 10, \cdots$)
\cr
h_n^{-1} \cdot h_n, 
&
($n = 3k + 2 = 5, 8, 11, \cdots$)
\cr
}
\]
By (\ref{asterisque}) we have
\[
h_{n+3} = \cases{
h_n (BA) \pi_{n+1} 
&
($n$ : odd)\cr
h_n (AB) \pi_{n+1} 
&
($n$ : even)\cr
}
\]
whose right limits coincide with 
$h_A$, $h_B$
and 
$h_{AA}$
respectively. 

We next study 
some substitutive properties of 
$v_{AA}$. 
Recall 
$h_{AA} \in \{ 0, 1 \}^{{\bf N}}$
is defined by the equation
$v_{AA} = h_{AA}^{-1} h_{AA}$. 
\begin{proposition}
(1)
$h_{AA}$
is the fixed point of the following substitution rule, 
\begin{eqnarray*}
\sigma : \quad&&
A \mapsto AB', A' \mapsto BA', 
\\
&&
B \mapsto A, B' \mapsto A'
\end{eqnarray*}
under the identification of 
$A$, $B$
with
$A'$, $B'$.
\\
(2)
Define the sequence of words
$\{ t_n \}_{n \ge 0}$ 
by 
\[
t_{n+1} = t_n \overline{t_{n-1}}, 
\;
n \ge 1, 
\quad
t_0 = B, \quad t_1 = A. 
\]
where
$\overline{s}$
is obtained by exchanging 
$A$, $B$
with
$A'$, $B'$
in 
$s^{-1}$. 
Let 
$t$ 
be the right limit of 
$t_n$
(we identify 
$A$, $B$
with 
$A'$, $B'$ 
in 
$t$). 
Then 
$t = h_{AA}$.
\end{proposition}
\begin{proof}
We can show 
$t_n = \sigma^n (A)$
by the inductive argument and the equation  
$\sigma(\overline{s}) = \overline{ \sigma (s)}$. 
$t = h_{AA}$
then follows from 
Lemma \ref{tandh} given below. 
\QED
\end{proof}
Let 
$t'_n$
be the word obtained by identifying 
$A$, $B$
with 
$A'$, $B'$
in 
$t_n$. 
\begin{lemma}
\label{tandh}
For 
$n \ge 1$
odd, we have
\begin{eqnarray*}
t'_{3n+2}
&=&
h_{3n+2} (BA) h_{3n+2}^{-1}
\\
t'_{3n+3}
&=&
h_{3n+2} (BA) \pi_{3n+1} (AB) h_{3n+2}^{-1}
\\
t'_{3n+4}
&=&
h_{3n+2} (BA) h_{3n+5}^{-1}
=
h_{3n+5} (AB) h_{3n+2}^{-1}
\end{eqnarray*}
(for even 
$n$, 
we exchange 
$AB$
with 
$BA$)
\end{lemma}
%

\section{Appenix 2 : robustness against local move}
In the 
Fibonacci case
($\alpha = \frac {1}{\tau}$), 
we can exchange 
$10$
with 
$01$
in 
$v \in \Omega$
at some site. 
A natural question
is whether it remains in the hull after this exchange. 
Let 
${\cal E}^{(i, i+1)}$
be this exchange operation at site 
$i, i+1$
(we always assume 
$v(i) \ne v(i+1)$). 
We can see
${\cal E}^{(-1, 0)}v_0 = v'_0$
which is, however, essentially the only case where this exchange is possible. 
\begin{theorem}
\label{local move}
Let 
$\alpha \in {\bf Q}^c \cap (0,1)$. 
If 
$v \in \Omega \setminus (\Omega_R \cup \Omega_L)$, 
then
${\cal E}^{(i, i+1)} v \notin  \Omega$
for any $i$.
\end{theorem}
As a preparation, we prove
\begin{lemma}
\label{n}
Let 
$v \in \Omega$.
If 
${\cal E}^{(m-1, m)}v \in \Omega$
for some 
$m$, 
then for any 
$n \ge 0$
the 
$(n-1, n)$-partition of 
$v$
has one of the following form. 
\begin{eqnarray*}
(a)\quad
&&
\fbox{\mbox{$s_{n-1}$}}
\fbox{ \mbox{　　$s_n$　　}}
\Bigg|
\fbox{ \mbox{　　$s_n$　　}}
\\
(b)\quad
&&
\fbox{ \mbox{　　$s_n$　　}}
\fbox{\mbox{$s_{n-1}$}}
\Bigg|
\fbox{ \mbox{　　$s_n$　　}}
\end{eqnarray*}
where 
$m$
is the site left to 
$\big |$.
Furthermore, if the 
$(n-1, n)$-partition
satisfies 
(a)
(resp. (b)), then the 
$(n, n+1)$-partition
satisfies 
(b)(resp. (a)), 
where 
$n$
is replaced by 
$n+1$. 
\end{lemma}
Lemma \ref{n}
is proved by induction. 
Then 
Theorem \ref{local move}
follows from the fact that
$v_0$
(resp. $v'_0$)
is the right limit of 
$s_n$
and the left limit of 
$s_{2n}$
(resp. $s_{2n+1}$). 
We can also consider exchanging 
$s_k$
with 
$s_{k-1}$
somewhere in the 
$(k-1, k)$-partition of 
$v$
and the same result as Theorem \ref{local move} holds.

\vspace*{1em}

\noindent {\bf Acknowledgement }
The author 
would like to thank professors 
Shigeki Akiyama, Kazushi Komatsu and Ms. Hiroko Hayashi for discussions. 
This work is partially supported by 
JSPS grant Kiban-C no.18540125.

%

\end{document}